\documentclass[11pt]{elsarticle}
\usepackage{amsmath,amssymb,amsthm}
\usepackage{hyperref}
\usepackage{mathrsfs}
\usepackage{graphicx}
\usepackage{tikz}
\usepackage{enumerate}
\usepackage{float}
\usepackage{bm}
\usepackage[all]{xy}

\theoremstyle{plain}
\newtheorem{theorem}{Theorem}[section]
\newtheorem{lemma}[theorem]{Lemma}

\newtheorem{proposition}[theorem]{Proposition}
\newtheorem{corollary}[theorem]{Corollary}

\theoremstyle{definition}
\newtheorem{definition}[theorem]{Definition}

\newtheorem{example}[theorem]{Example}

\newcommand{\ua}{\mathord{\uparrow}}
\newcommand{\da}{\mathord{\downarrow}}
\newcommand{\rom}[1]{\rm{\uppercase\expandafter{\romannumeral #1}}}

\makeatletter
\def\ps@pprintTitle{%
  \let\@oddhead\@empty
  \let\@evenhead\@empty
  \def\@oddfoot{\reset@font\hfil\thepage\hfil}
  \let\@evenfoot\@oddfoot
}
\makeatother

\begin{document}

\begin{frontmatter}

\title{On some open problems concerning strong $d$-spaces and super $\mathrm{H}$-sober spaces\tnoteref{t1}}
\tnotetext[t1]{This work is supported by the National Natural Science Foundation of China (No.11771134) and by Hunan Provincial Innovation Foundation For Postgraduate (CX20200419)}

\author{Mengjie Jin}
\ead{mengjiejinjin@163.com}
\author{Hualin Miao\corref{a1}}
\ead{miaohualinmiao@163.com}
\author{Qingguo Li\corref{a1}}
\ead{liqingguoli@aliyun.com}
\address{School of Mathematics, Hunan University, Changsha, Hunan, 410082, China}
\cortext[a1]{Corresponding author.}

\begin{abstract}
In this paper, we highlight some open problems stated by Xu and Zhao. In particular, we focus on strong $d$-spaces and answer two open problems concerning strong $d$-spaces. One is about the product space of an arbitrary family of strong $d$-spaces. The other concerns the reflectivity of the category $\mathbf{S}$-$\mathbf{Top}_{d}$ of strong $d$-spaces. Meantime, for an irreducible subset system ($R$-subset system for short) $\mathrm{H}$ and a $T_{0}$ space $X$, we prove that $\mathrm{H}$ naturally satisfies property $M$. Based on this, we deduce that $X$ is super $\mathrm{H}$-sober iff $X$ is $\mathrm{H}$-sober and $\mathrm{H}$ satisfies property $Q$. Additionally, we give positive answers to some questions posed by Xu in \cite{Xiaoquan21}. Furthermore, we obtain that the finite product of hyper-sober spaces is hyper-sober.
\end{abstract}
\end{frontmatter}

\section{Introduction}
Sobriety,  monotone convergence and well-filteredness are three of the most important and useful properties in non-Hausdorff topological spaces and domain theory (see \cite{clad3},\cite{nht2},\cite{R. Heckmann} and \cite{Zhao Xu}). In recent years, sober spaces, monotone convergence spaces (shortly called $d$-spaces), well-filtered spaces and their related structures have been introduced and investigated. In order to uncover more finer links between $d$-spaces and well-filtered spaces, the concept of strong $d$-spaces was proposed in \cite{Zhao Xu}. And in \cite{Xiaoquan}, Xu and Zhao posed the following two questions:

{\bf Question 1.1.} Is the product space of arbitrary family of strong $d$-spaces again a strong $d$-space?

{\bf Question 1.2.} Is $\mathbf{S}$-$\mathbf{Top}_{d}$ reflective in $\mathbf{Top}_{0}$?

In \cite{Xiaoquan21}, Xu introduced a uniform approach to sober spaces, $d$-spaces and well-filtered spaces and developed a general frame for dealing with all these spaces. The concepts of irreducible subset systems ($R$-subset systems for short), $\mathrm{H}$-sober spaces and super $\mathrm{H}$-sober spaces for an $R$-subset system $\mathrm{H}$ were proposed. Let $\mathbf{Top}_{0}$ be the category of all $T_{0}$ spaces and $\mathbf{Sob}$ the category of all sober spaces. The category of all $\mathrm{H}$-sober spaces (resp., super $\mathrm{H}$-sober spaces) with continuous mappings is denoted by $\mathbf{H}$-$\mathbf{Sob}$ (resp., $\mathbf{SH}$-$\mathbf{Sob}$). For any $R$-subset system $\mathrm{H}$, it has been proven that $\mathbf{H}$-$\mathbf{Sob}$ is a full subcategory  of $\mathbf{Top}_{0}$ containing $\mathbf{Sob}$ and is closed with respect to homeomorphisms. Moreover, $\mathbf{H}$-$\mathbf{Sob}$ is adequate (see Theorem 7.9 in \cite{Xiaoquan21}). Conversely, for a full subcategory $\mathbf{K}$ of $\mathbf{Top}_{0}$ containing $\mathbf{Sob}$, suppose that $\mathbf{K}$ is adequate and closed with respect to homeomorphisms. There is a natural question: Is there an R-subset system $\mathrm{H}$ such that $\mathbf{K}=\mathbf{H}$-$\mathbf{Sob}$? In this paper, we will prove the answer is positive. In \cite{Xiaoquan21}, Xu also proposed an $R$-subset system $\mathrm{H}:\mathbf{Top}_{0}\longrightarrow \mathbf{Set}$ is said to satisfy property $M$ if for any $T_{0}$ space $X$, $\mathcal{K}\in \mathrm{H}(P_{S}(X))$ and $A\in M(\mathcal{K})$, then $\{\ua (K\cap A)\mid K\in \mathcal{K}\}\in \mathrm{H}(P_{S}(X))$. Furthermore, he obtained some results under the assumption that $\mathrm{H}$ has property $M$. Then he posed the following questions:

{\bf Question 1.3.} Let $\mathrm{H}:\mathbf{Top}_{0}\longrightarrow \mathbf{Set}$ be an $R$-subset system (may not have property $M$) and $\{X_{i}\mid i\in I\}$ a
family of super $\mathrm{H}$-sober spaces. Is the product space $\prod_{i\in I}X_{i}$ super $\mathrm{H}$-sober?

{\bf Question 1.4.} Let $\mathrm{H}:\mathbf{Top}_{0}\longrightarrow \mathbf{Set}$ be an $R$-subset system, $X$ a $T_{0}$ space and $Y$ a super $\mathrm{H}$-sober space. Is the function space $\mathbf{Top}_{0}(X,Y)$ equipped with the topology of pointwise convergence super $\mathrm{H}$-sober?

{\bf Question 1.5.} Let $\mathrm{H}:\mathbf{Top}_{0}\longrightarrow \mathbf{Set}$ be an $R$-subset system, $X$ a super $\mathrm{H}$-sober space and $Y$ a $T_{0}$ space. For a pair of continuous mappings $f,g:X\rightarrow Y$, is the equalizer $E(f,g)=\{x\in X\mid f(x)=g(x)\}$ (as a
subspace of $X$) super $\mathrm{H}$-sober?

{\bf Question 1.6.} For an $R$-subset system $\mathrm{H}:\mathbf{Top}_{0}\longrightarrow \mathbf{Set}$, is $\mathrm{H}$ $\mathbf{SH}$-$\mathbf{Sob}$ complete?

{\bf Question 1.7.} If an $R$-subset system $\mathrm{H}:\mathbf{Top}_{0}\longrightarrow \mathbf{Set}$ has property $M$, do the induced $R$-subset systems $\mathrm{H}^{d}$, $\mathrm{H}^{R}$ and $\mathrm{H}^{D}$ have property M?

{\bf Question 1.8.} For an $R$-subset system $\mathrm{H}:\mathbf{Top}_{0}\longrightarrow \mathbf{Set}$, is $\mathbf{SH}$-$\mathbf{Sob}$ reflective in $\mathbf{Top}_{0}$? Or equivalently,
for any $T_{0}$ space $X$, does the super $\mathrm{H}$-sobrification of $X$ exist?

{\bf Question 1.9.} Let $\mathrm{H}:\mathbf{Top}_{0}\longrightarrow \mathbf{Set}$ be an $R$-subset system having property $M$. Is $\mathrm{R}:\mathcal{H}\rightarrow \mathcal{H}$, $\mathrm{H}\mapsto \mathrm{H}^{R}$ a closure operator?

In \cite{Zhao15}, Zhao and Ho introduced a new variant of sobriety, called hyper-sobriety. A topological space $X$ is called hyper-sober if for any irreducible set $F$, there is a unique $x\in F$ such that $F\subseteq \mathrm{cl}(\{x\})$. Clearly, every hyper-sober space is sober. Later, Wen and Xu discussed some basic properties of hyper-sober spaces in \cite{Wennana}. Moreover, they posed a question:

{\bf Question 1.10.} Is the product space of two hyper-sober spaces again a hyper-sober space?

In Section 3, we will give negative answers to Question 1.1 and Question 1.2.

In Section 4, we find that property $M$ mentioned above naturally holds for each $R$-subset system $\mathrm{H}$. Based on this, for a $T_{0}$ space $X$, we deduce that $X$ is super $\mathrm{H}$-sober iff $X$ is $\mathrm{H}$-sober and $\mathrm{H}$ satisfies property $Q$. Additionally, we give positive answers to Question 1.3$\sim$Question 1.8. Furthermore, we prove Question 1.9 holds.

In Section 5, we will give a positive answer to Question 1.10.
\section{Preliminaries}
In this section, we briefly recall some standard definitions and notations to be used in this paper. For further details, refer to \cite{clad3,nht2,R. Heckmann,Heckmann,Xiaoquan20}.

Let $P$ be a poset and $A\subseteq P$. We denote $\da A=\{x\in P \mid x\leq a \mbox{ for some } a\in A\}$ and $\ua A=\{x\in P \mid x\geq a \mbox{ for some } a\in A\}$. For any $a\in P$, we denote $\da\{a\}=\da a=\{x\in P \mid x\leq a\}$
and $\ua\{a\}=\ua a=\{x\in P \mid x\geq a\}$. $\da a$ and $\ua a$ are called \emph{principal ideal} and \emph{principal filter} respectively. A subset $A$ is called a \emph{lower set} (resp., an \emph{upper set}) if $A=\da A$ (resp., $A=\ua A$).

In \cite{clad3}, a subset $D$ of poset $P$ is \emph{directed} provided that it is nonempty and every finite subset of $D$ has an upper bound in $D$. The set of all directed sets of $P$ is denoted by $\mathcal{D}(P)$. The poset $P$ is called a \emph{dcpo} if every directed subset $D$ in $P$ has a supremum. A subset $U$ of $P$ is \emph{Scott open} if (1)$U=\ua U$ and (2) for any directed subset $D$ for which $\vee D$ exists, $\vee D\in U$ implies $D\cap U\neq \emptyset$. All Scott open subsets of $P$ from a topology, and we call this topology the \emph{Scott topology} on $P$ and denoted by $\sigma(P)$. The \emph{upper topology} on $P$ , generated by the complements of the principal ideals of $P$, is denoted by $\upsilon(P)$.

The category of all $T_{0}$ spaces with continuous mappings is denoted by $\mathbf{Top}_{0}$. For a $T_{0}$ space $X$, let $\mathcal{O}(X)$ (resp., $\Gamma(X)$) be the set of all open subsets (resp., closed subsets) of $X$. For a subset $A$ of $X$, the closure of $A$ is denoted by $\mathrm{cl}(A)$ or $\overline{A}$. We use $\leq_{X}$ to represent the specialization order of $X$, that is, $x\leq_{X}y$ iff $x\in \overline{\{y\}}$. A subset $B$ of $X$ is called \emph{saturated} if $B$ equals the intersection of all open sets containing it (equivalently, $B$ is an upper set in the specialization order). Let $S(X)=\{\{x\}\mid x\in X\}$ and $S_{c}(X)=\{\overline{\{x\}}\mid x\in X\}$.
A $T_{0}$ space $X$ is called a \emph{$d$-space} (i.e., \emph{monotone convergence space}) if $X$ (with the specialization order) is a $dcpo$ and $\mathcal{O}(X)\subseteq \sigma(X)$ (\cite{clad3}).

For a $T_{0}$ space $X$ and a nonempty subset $A$ of $X$, $A$ is called \emph{irreducible} if for any $B,C\in \Gamma(X)$, $A\subseteq B\cup C$ implies $A\subseteq B$ or $A\subseteq C$. Denote by $\mathrm{Irr}(X)$ the set of all irreducible subsets of $X$ and $\mathrm{Irr}_{c}(X)$ the set of all closed irreducible subsets of $X$. A topological space $X$ is called \emph{sober}, if for any $A\in \mathrm{Irr}_{c}(X)$, there is a unique point $a\in X$ such that $A=\overline{\{a\}}$. The category of all sober spaces with continuous mappings is denoted by $\mathbf{Sob}$.

For a topological space $X$, $\mathcal{G}\subseteq 2^{X}$ and $A\subseteq X$, let $\Box_{\mathcal{G}}(A)=\{G\in \mathcal{G}\mid G\subseteq A\}$ and $\diamondsuit_{\mathcal{G}}(A)=\{G\in \mathcal{G}\mid G\cap A\neq \emptyset\}$. The symbols $\Box_{\mathcal{G}}(A)$ and $\diamondsuit_{\mathcal{G}}(A)$ will be simply written as $\Box A$ and $\diamondsuit A$ respectively if there is no
confusion. The \emph{upper Vietoris topology} on $\mathcal{G}$ is the topology that has $\{\Box U\mid U\in \mathcal{O}(X)\}$ as a base, and the resulting space is denoted by $P_{S}(\mathcal{G})$. The \emph{lower Vietoris topology} on $\mathcal{G}$ is the topology that has $\{\diamondsuit U\mid U\in \mathcal{O}(X)\}$ as a subbase, and the resulting space is denoted by $P_{H}(\mathcal{G})$. If $\mathcal{G}\subseteq \mathrm{Irr}(X)$, then $\{\diamondsuit U\mid U\in \mathcal{O}(X)\}$ is a topology on $\mathcal{G}$.

In what follows, The category $\mathbf{K}$ always refers to a full subcategory $\mathbf{Top}_{0}$ containing $\mathbf{Sob}$, the objects of $\mathbf{K}$ are called \emph{$\mathbf{K}$-spaces}.

\begin{definition}\label{k-reflection}\cite{Xiaoquan20}
Let $X$ be a $T_{0}$ space. A \emph{$\mathbf{K}$-reflection} of $X$ is a pair $\langle \widehat{X},\mu\rangle$ consisting of a $\mathbf{K}$-space $\widehat{X}$ and a continuous mapping $\mu:X\rightarrow \widehat{X}$ satisfying that for any continuous mapping $f:X\rightarrow Y$ to a $\mathbf{K}$-space, there exists a unique continuous mapping $f^{\ast}:\widehat{X}\rightarrow Y$ such that $f^{\ast}\circ \mu= f$, that is, the following diagram commutes.
\end{definition}
\centerline{\xymatrix{X \ar[dr]_{f} \ar[r]^{\mu}
                & \widehat{X} \ar[d]^{f^{\ast}}  \\
                &  Y }}

By a standard argument, $\mathbf{K}$-reflections, if they exist, are unique to homeomorphisms. We shall use $X^{k}$ to indicate the space of the $\mathbf{K}$-reflection of $X$ if it exists.

\begin{definition}\cite{Xiaoquan20}\label{k-determined}
Let $(X,\tau)$ be a $T_{0}$ space. A nonempty subset $A$ of $X$ is called \emph{$\mathbf{K}$-determined} provided for any continuous mapping $f: X\rightarrow Y$ to a $\mathbf{K}$-space $Y$, there exists a unique $y_{A}\in Y$ such that $\mathrm{cl}_{Y}(f(A))=\mathrm{cl}_{Y}(\{y_{A}\})$. Clearly, a
subset $A$ of a space $X$ is a $\mathbf{K}$-determined set iff $\mathrm{cl}(A)$ is a K-determined set. Denote by $\mathrm{K}^{d}(X)$ the set of all $\mathbf{K}$-determined subsets of $X$ and $\mathrm{K}^{d}_{c}(X)$ the set of all closed $\mathbf{K}$-determined subsets of $X$.
\end{definition}

\begin{lemma}\label{1}
Let $X$, $Y$ be two $T_{0}$ spaces. If $f:X\rightarrow Y$ is a continuous mapping and $A\in \mathrm{K}^{d}(X)$, then $f(A)\in \mathrm{K}^{d}(Y)$.
\end{lemma}

\begin{proof}
Let $Z$ be a $\mathbf{K}$-space and $g:Y\rightarrow Z$ is continuous. Note that $g\circ f:X\rightarrow Z$ is continuous and $A\in \mathrm{K}^{d}(X)$, then there exists an element $z\in Z$ such that $\overline{g\circ f(A)}=\overline{g(f(A))}=\overline{\{z\}}$. So $f(A)\in \mathrm{K}^{d}(Y)$.
\end{proof}

\begin{definition}\cite{Xiaoquan20}
A full subcategory $\mathbf{K}$ of $\mathbf{Top}_{0}$ is said to be \emph{closed with respect to homeomorphisms} if homeomorphic copies of $\mathbf{K}$-spaces are $\mathbf{K}$-spaces.
\end{definition}

\begin{definition}\cite{Xiaoquan20}
$\mathbf{K}$ is called \emph{adequate} if for any $T_{0}$ space $X$, $P_{H}(\mathrm{K}^{d}_{c}(X))$ is a $\mathbf{K}$-space.
\end{definition}

\begin{corollary}\cite{Xiaoquan20}\label{3}
 Assume that $\mathbf{K}$ is adequate and closed with respect to homeomorphisms. Then for any $T_{0}$ space $X$, the following conditions are equivalent:
 \begin{enumerate}[(1)]
 \item $X$ is a $\mathbf{K}$-space.
 \item $\mathrm{K}^{d}_{c}(X)=S_{c}(X)$, that is, for each $A\in \mathrm{K}^{d}_{c}(X)$, there exists an $x\in X$ such that $A=\overline{\{x\}}$.
 \item $X\cong X^{k}$.
 \end{enumerate}
\end{corollary}

\begin{definition}\cite{Davey02}\label{6}
Let $P$ be an ordered set. A map $c:P\rightarrow P$ is called a \emph{closure operator}(on $P$) if, for all $x,y\in P$,
\begin{enumerate}[(1)]
\item $x\leq y \Longrightarrow c(x)\leq c(y)$ ($c$ is monotone),
\item $x\leq c(x)$ ($c$ is expansive),
\item $c(c(x))=c(x)$ ($c$ is idempotent).
\end{enumerate}
\end{definition}

We shall use $K(X)$ to denote the set of all nonempty compact
saturated subsets of $X$ ordered by reverse inclusion, that is, for $K_{1},K_{2}\in K(X)$, $K_{1}\leq_{K(X)}K_{2}$ iff $K_{2}\subseteq K_{1}$. The space $P_{S}(K(X))$ denoted shortly by $P_{S}(X)$, is called the \emph{Smyth power space} or \emph{upper space} of $X$ (\cite{Heckmann}).

\begin{lemma}\cite{X Jia}\label{9}
Let $X$ be an $T_{0}$ space. If $\mathcal{K}\in K(P_{S}(X))$, then $\bigcup \mathcal{K}\in K(X)$.
\end{lemma}

\begin{corollary}\cite{X Jia}\label{8}
For  a $T_{0}$ space $X$, the mapping $\bigcup:P_{S}(P_{S}(X))\rightarrow P_{S}(X)$, $\mathcal{K}\mapsto \bigcup\mathcal{K}$, is continuous.
\end{corollary}

\section{On the two questions about strong $d$-spaces}
In order to uncover more finer links between $d$-spaces and well-filtered spaces, the concept of strong $d$-spaces was introduced in \cite{Zhao Xu}. And in \cite{Xiaoquan}, Xu and Zhao posed the following two questions:

{\bf Question 1}: Is the product space of arbitrary family of strong $d$-spaces again a strong $d$-space?

{\bf Question 2}: Is $\mathbf{S}$-$\mathbf{Top}_{d}$ reflective in $\mathbf{Top}_{0}$?

In this Section, we will give negative answers to the above questions.

\begin{definition}(\cite{Zhao Xu,Xiaoquan})
A $T_{0}$ space $X$ is called a \emph{strong $d$-space} if for any directed subset $D$ of $X$, $x\in X$ and $U\in \mathcal{O}(X)$, $\bigcap_{d\in D}\ua d\cap \ua x\subseteq U$ implies $\ua d\cap \ua x\subseteq U$ for some $d\in D$. The category of all strong $d$-spaces with continuous mappings is denoted by $\mathbf{S}$-$\mathbf{Top}_{d}$.
\end{definition}

First, we supply a negative answer to Question 1.

\begin{example}\label{a}
 Let $L=\mathbb{N} \cup \{a_{n}\mid n\in \mathbb{N}\}\cup\{\beta_{n}\mid n\in \mathbb{N}\}\cup \{\omega, a\}
  $, where $\mathbb{N}=\{1,2,3,\cdots\}$. We define an order $\leq$ on $L$ as follows:

   $x\leq y$ if and only if:

  $\bullet$ $x\leq y$ in $\mathbb{N}$;

  $\bullet$ $x\in \mathbb{N}\cup\{\omega\}$, $y=\omega$;

  $\bullet$ $x=a$, $y\in \{\beta_{n}\mid n\in \mathbb{N}\}\cup \{a_{n}\mid n\in \mathbb{N}\}$;

  $\bullet$ $x\in \mathbb{N}$, $y\in\{a_{n}\mid n\geq x\}$;

  $\bullet$ $x\in \mathbb{N}\cup \{\omega\}$, $y\in \{\beta_{n}\mid n\in \mathbb{N}\}$.

  Then $L$ can be easily depicted as in Fig. 1. Let $\tau_{1}=\{U\in \sigma(L)\mid \{a_{n}\mid n\in \mathbb{N}\}\backslash U~\mathrm{is}~\mathrm{finite}\}$ and $\tau_{2}=\{U\in \sigma(L)\mid U\subseteq \{a_{n}\mid n\in \mathbb{N}\}\}$. Define $\tau_{L}=\tau_{1}\cup \tau_{2}$. Then $(L,\tau_{L})$ is a strong $d$-space. However, $L\times L$ is not a strong $d$-space.
\end{example}

  \begin{figure}[H]
\centering
\begin{tikzpicture}[scale=1.0]
\path (0,1)   node[left] {1} coordinate (a);
\fill (a) circle (1pt);
\path (0,2)   node[left] {2} coordinate (b);
\fill (b) circle (1pt);
\path (0,3)   node[left] {3} coordinate (c);
\fill (c) circle (1pt);
\path (0,6)   node[right] {$\omega$} coordinate (d);
\fill (d) circle (1pt);
\path (-1,7)  node[above] {$\beta_{1}$} coordinate (e);
\fill (e) circle (1pt);
\path (0,7)  node[above] {$\beta_{2}$} coordinate (f);
\fill (f) circle (1pt);
\path (1,7)  node[above] {$\beta_{3}$} coordinate (g);
\fill (g) circle (1pt);
\path (1.2,7)  node[right] {$$} coordinate (n);
\fill (n) circle (0.1pt);
\path (2.2,7)  node[right] {$$} coordinate (o);
\fill (o) circle (0.1pt);
\path (3,7)  node[above] {$a_{1}$} coordinate (h);
\fill (h) circle (1pt);
\path (4,7)  node[above] {$a_{2}$} coordinate (i);
\fill (i) circle (1pt);
\path (5,7)  node[above] {$a_{3}$} coordinate (j);
\fill (j) circle (1pt);
\path (4,1)  node[right] {$a$} coordinate (k);
\fill (k) circle (1pt);
\path (5.5,7)  node[right] {$$} coordinate (l);
\fill (l) circle (0.1pt);
\path (7,7)  node[right] {$$} coordinate (m);
\fill (m) circle (0.1pt);
\draw  (a) -- (b)  (b) -- (c) (d) -- (e) (d) -- (f) (d) -- (g) (a) -- (h) (b) -- (i) (c) -- (j) (k) -- (e) (k) -- (f) (k) -- (g) (k) -- (h) (k) -- (i) (k) -- (j);
\draw[densely dashed] (c)--(d) (l)--(m) (n)--(o);
\end{tikzpicture}
\\Fig. 1. $L$.
\end{figure}

  \begin{proof}
   {\bf Claim 1}: $L$ is a $dcpo$.

   It is easy to verify that $L$ is a $dcpo$.

   {\bf Claim 2}: $\tau_{L}$ is a topology and $\nu(L)\subseteq \tau_{L}\subseteq \sigma(L)$.

    Obviously, $\emptyset,L\in \tau_{L}$. Let $U,V\in \tau_{L}$. If $U,V\in \tau_{1}$, then $U\cap V\in \sigma(L)$ and $\{a_{n}{\mid} n\in \mathbb{N}\}\backslash (U\cap V)=(\{a_{n}{\mid} n\in \mathbb{N}\}\backslash U)\cup (\{a_{n}{\mid} n\in \mathbb{N}\}\backslash V)$ is finite since $U,V\in \tau_{1}$. This means that $U\cap V\in \tau_{1}\subseteq \tau_{L}$. If $U\in \tau_{1},V\in \tau_{2}$, then $U\cap V\in \tau_{2}$. If $U,V\in \tau_{2}$, then obviously, $U\cap V\in \tau_{2}$. Assume $(U_{i})_{i\in I}\subseteq \tau_{L}$. If there exists $i_{0}\in I$ such that $i_{0}\in \tau_{1}$, then $\bigcup_{i\in I}U_{i}\in \tau_{1}$. Else, $(U_{i})_{i\in I}\subseteq \tau_{2}$, then we have $\bigcup_{i\in I}U_{i}\in \tau_{2}$. Hence, $\tau_{L}$ is a topology.

     It is straightforward to check that $\nu(L)\subseteq \tau_{L}\subseteq \sigma(L)$.

     Thus we have that $(L,\tau_ L)$ is a $d$-space.

   {\bf Claim 3}: $(L,\tau_{L})$ is a strong $d$-space.

    It suffices to prove that for any infinite chain $C$ in $\mathbb{N}$, $x\in L$ and $U\in \tau_{L}$, $\bigcap_{c\in C}\ua c\cap \ua x\subseteq U$ implies that there exists $c\in C$ such that $\ua c\cap\ua x\subseteq U$. Now we need to distinguish the following four cases:

    Case 1, $x\in \mathbb{N}\cup\{\omega\}$. Then $\bigcap_{c\in C}\ua c\cap \ua x=\ua \omega \cap \ua x=\ua \omega\subseteq U$. Since $U$ is Scott open, there exists an element $c\in C$ such $c\in U$. This implies that $\ua c\cap \ua x\subseteq \ua c \subseteq U$.

    Case 2, $x\in \{\beta_{n} \mid n\in \mathbb{N}\}$. Without loss of generality, suppose $x=\beta_{n}$. Then $\bigcap_{c\in C}\ua c\cap \ua x=\ua \omega\cap \ua \beta_{n}=\{\beta_{n}\}\subseteq U$. And we find that for any $c\in C$, $\ua c\cap \ua \beta_{n}=\{\beta_{n}\}\subseteq U$.

    Case 3, $x\in \{a_{n}\mid n\in \mathbb{N}\}$. Assume $x=a_{n}$. Then $\bigcap_{c\in C}\ua c\cap \ua x=\ua \omega\cap \ua a_{n}=\emptyset$. Note that for any $n$, there exists an element $c\in C$ such that $n<c$, then $\ua c\cap \ua a_{n}=\emptyset\subseteq U$.

    Case 4, $x=a$. Then $\bigcap_{c\in C}\ua c\cap \ua x=\ua \omega \cap \ua a=\{\beta_{n}\mid n\in \mathbb{N}\}\subseteq U$. This implies that $U\in \tau_{1}$. It follows that there exists $n\in \mathbb{N}$ such that $\ua n \cap \ua a\subseteq U$. For $n$, since $C$ is infinite, there is an element $c\in C$ such that $n\leq c$. Therefore, $\ua c\cap \ua a \subseteq\ua n \cap \ua a\subseteq U$.

   {\bf Claim 4}: $L\times L$ is not a strong $d$-space.

    Let $U_{n}=\{\beta_{m}\mid m\leq n\}\cup\{a_{i}\mid i\geq n\}$. Obviously, $U_{n}\in \tau_{L}$. Let $W=\bigcup_{n\in\mathbb{N}}(U_{n}\times U_{n})\cup\bigcup_{n\in \mathbb{N}}(U_{n}\times \{a_{m}\mid m\leq n\})$. It is easy to show that $W$ is open in $L\times L$. We want to prove that $\bigcap_{n\in \mathbb{N}}\ua (n,1)\cap \ua (a,a)=\ua(\omega,1)\cap \ua (a,a)=\{\beta_{n}\mid n\in \mathbb{N}\}\times (\{\beta_{n}\mid n\in \mathbb{N}\}\cup\{a_{n}\mid n\in \mathbb{N}\})\subseteq W$. For any $m,n\in \mathbb{N}$, if $n\leq m$, then $(\beta_{n},\beta_{m})\in U_{m}\times U_{m}\subseteq W$, $(\beta_{n},a_{m})\in U_{n}\times U_{n}\subseteq W$. Else $n>m$, then $(\beta_{n},\beta_{m})\in U_{n}\times U_{n}\subseteq W$, $(\beta_{n},a_{m})\in U_{n}\times \{a_{i}\mid i\leq n\}\subseteq W$. Assume $L\times L$ is a strong $d$-space. Then there exists $n_{0}\in \mathbb{N}$ such that $\ua (n_{0},1)\cap \ua (a,a)=(\{\beta_{n}\mid n\in \mathbb{N}\}\cup \{a_{m}\mid m\geq n_{0}\})\times (\{\beta_{n}\mid n\in \mathbb{N}\}\cup \{a_{n}\mid n\in \mathbb{N}\})\subseteq W$, which contradicts the fact that $(a_{n_{0}},\beta_{n_{0}+1})\notin W$. Thus $L\times L$ is not a strong $d$-space.
 \end{proof}

In the following, we provide a negative answer to Question 2. We show that the category $\mathbf{S}$-$\mathbf{Top}_{d}$ is not a reflective subcategory of $\mathbf{Top}_{0}$.

A full subcategory $\mathcal{D}$ of a category $\mathcal{C}$ is called \emph{reflective} if the inclusion functor has a left
adjoint, which then is called a reflector(\cite{Barr}).

\begin{definition}\label{-reflection}
Let $X$ be a $T_{0}$ space. A  \emph{strong $d$-reflection} of $X$ is a pair $\langle \widehat{X},\mu\rangle$ consisting of a strong $d$-space $\widehat{X}$ and a continuous mapping $\mu:X\rightarrow \widehat{X}$ satisfying that for any continuous mapping $f:X\rightarrow Y$ to a strong $d$-space, there exists a unique continuous mapping $f^{\ast}:\widehat{X}\rightarrow Y$ such that $f^{\ast}\circ \mu= f$, that is, the following diagram commutes.
\end{definition}
\centerline{\xymatrix{X \ar[dr]_{f} \ar[r]^{\mu}
                & \widehat{X} \ar[d]^{f^{\ast}}  \\
                &  Y }}

By a standard argument,  strong $d$-reflections, if they exist, are unique to homeomorphisms. We shall use $X^{s-d}$ to indicate the space of the strong $d$-reflection of $X$ if it exists.

\begin{example}\label{c}
  Let $P=\mathbb{N} \cup \{a_{n}\mid n\in \mathbb{N}\}\cup \{\beta,\omega, a\}
  $, where $\mathbb{N}=\{1,2,3,\cdots\}$. We define an order $\leq$ on $P$ as follows:

   $x\leq y$ if and only if:

  $\bullet$ $x\leq y$ in $\mathbb{N}$;

  $\bullet$ $x\in \mathbb{N}\cup\{\omega\}$, $y=\omega$;

  $\bullet$ $x=a$, $y\in \{\beta\}\cup \{a_{n}\mid n\in \mathbb{N}\}$;

  $\bullet$ $x\in \mathbb{N}$, $y\in\{a_{n}\mid n\geq x\}$;

  $\bullet$ $x\in \mathbb{N}\cup \{\omega\}$, $y\in \{\beta\}$.

Then $P$ can be easily depicted as in Fig. 2. Define $\tau_{P}=\{U{\in} \sigma(P){\mid } \{a_{n}{\mid} n\in \mathbb{N}\}\backslash U ~\mathrm{is} ~\mathrm{finite}\}\cup\{\emptyset\}$. Then $(P,\tau_{P})$ is a strong $d$-space.
\begin{proof}
  The proof is similar to that of Example \ref{a}.
\end{proof}
\end{example}
\begin{figure}[H]
\centering
\begin{tikzpicture}[scale=1.0]
\path (0,1)   node[left] {1} coordinate (a);
\fill (a) circle (1pt);
\path (0,2)   node[left] {2} coordinate (b);
\fill (b) circle (1pt);
\path (0,3)   node[left] {3} coordinate (c);
\fill (c) circle (1pt);
\path (0,6)   node[left] {$\omega$} coordinate (d);
\fill (d) circle (1pt);
\path (0,7)  node[left] {$\beta$} coordinate (f);
\fill (f) circle (1pt);
\path (3,7)  node[right] {$a_{1}$} coordinate (h);
\fill (h) circle (1pt);
\path (4,7)  node[right] {$a_{2}$} coordinate (i);
\fill (i) circle (1pt);
\path (5,7)  node[right] {$a_{3}$} coordinate (j);
\fill (j) circle (1pt);
\path (4,1)  node[right] {$a$} coordinate (k);
\fill (k) circle (1pt);
\path (5.5,7)  node[right] {$$} coordinate (l);
\fill (l) circle (0.1pt);
\path (7,7)  node[right] {$$} coordinate (m);
\fill (m) circle (0.1pt);
\draw  (a) -- (b)  (b) -- (c) (d) -- (f) (a) -- (h) (b) -- (i) (c) -- (j) (k) -- (f) (k) -- (h) (k) -- (i) (k) -- (j);
\draw[densely dashed] (c)--(d) (l)--(m);
\end{tikzpicture}
\\Fig. 2. $P$.
\end{figure}
\begin{example}\label{b}
  Let $Q=\mathbb{N} \cup \{a_{n}\mid n\in \mathbb{N}\}\cup \{\omega, a\}
  $, where $\mathbb{N}=\{1,2,3,\cdots\}$. We define an order $\leq$ on $P$ as follows:

   $x\leq y$ if and only if:

  $\bullet$ $x\leq y$ in $\mathbb{N}$;

  $\bullet$ $x\in \mathbb{N}\cup\{\omega\}$, $y=\omega$;

  $\bullet$ $x=a$, $y\in \{a_{n}\mid n\in \mathbb{N}\}$;

  $\bullet$ $x\in \mathbb{N}$, $y\in\{a_{n}\mid n\geq x\}$.

Then $Q$ can be easily depicted as in Fig. 3. Define $\tau_{Q}{=}\{U{\in} \sigma(Q){\mid} \{a_{n}{\mid} n\in \mathbb{N}\}\backslash U ~\mathrm{is} ~\mathrm{finite}\}\cup\{\emptyset\}$. Then $(Q,\tau_{Q})$ is not a strong $d$-space.
\begin{proof}
  Note that $\bigcap_{n\in \mathbb{N}}\ua n\cap \ua a=\ua \omega \cap \ua a=\emptyset$. But $\ua a\cap \ua n\neq \emptyset$ for any $n\in \mathbb{N}$. Hence, $(Q,\tau_{Q})$ is not a strong $d$-space.
\end{proof}
\end{example}
\begin{figure}[H]
\centering
\begin{tikzpicture}[scale=1.0]
\path (0,1)   node[left] {1} coordinate (a);
\fill (a) circle (1pt);
\path (0,2)   node[right] {2} coordinate (b);
\fill (b) circle (1pt);
\path (0,3)   node[right] {3} coordinate (c);
\fill (c) circle (1pt);
\path (0,7)   node[right] {$\omega$} coordinate (d);
\fill (d) circle (1pt);

\path (3,7)  node[right] {$a_{1}$} coordinate (h);
\fill (h) circle (1pt);
\path (4,7)  node[right] {$a_{2}$} coordinate (i);
\fill (i) circle (1pt);
\path (5,7)  node[right] {$a_{3}$} coordinate (j);
\fill (j) circle (1pt);
\path (4,1)  node[right] {$a$} coordinate (k);
\fill (k) circle (1pt);
\path (5.5,7)  node[right] {$$} coordinate (l);
\fill (l) circle (0.1pt);
\path (7,7)  node[right] {$$} coordinate (m);
\fill (m) circle (0.1pt);
\draw  (a) -- (b)  (b) -- (c) (a) -- (h) (b) -- (i) (c) -- (j) (k) -- (h) (k) -- (i) (k) -- (j);
\draw[densely dashed] (c)--(d) (l)--(m);
\end{tikzpicture}
\\Fig. 3. $Q$.
\end{figure}

\begin{example}\label{d}
Let $M=\mathbb{N} \cup \{a_{n}\mid n\in \mathbb{N}\}\cup \{\alpha,\beta,\omega, a\}$, where $\mathbb{N}=\{1,2,3,\cdots\}$. We define an order $\leq$ on $M$ as follows:

   $x\leq y$ if and only if:

  $\bullet$ $x\leq y$ in $\mathbb{N}$;

  $\bullet$ $x\in \mathbb{N}\cup\{\omega\}$, $y=\omega$;

  $\bullet$ $x=a$, $y\in \{\beta,\alpha\}\cup \{a_{n}\mid n\in \mathbb{N}\}$;

  $\bullet$ $x\in \mathbb{N}$, $y\in\{a_{n}\mid n\geq x\}$;

  $\bullet$ $x\in \mathbb{N}\cup \{\omega\}$, $y\in \{\beta,\alpha\}$.
Then $M$ can be easily depicted as in Fig. 4. Define $\tau_{M}=\{U\in \sigma(M)\mid \{a_{n}\mid n\in \mathbb{N}\}\backslash U ~\mathrm{is} ~\mathrm{finite}\}\cup\{\emptyset\}$. Then $(M,\tau_{M})$ is a strong $d$-space.
\begin{proof}
 The proof is similar to that of Example \ref{a}.
\end{proof}
\end{example}
\begin{figure}[H]
\centering
\begin{tikzpicture}[scale=1.0]
\path (0,1)   node[left] {1} coordinate (a);
\fill (a) circle (1pt);
\path (0,2)   node[left] {2} coordinate (b);
\fill (b) circle (1pt);
\path (0,3)   node[left] {3} coordinate (c);
\fill (c) circle (1pt);
\path (0,6)   node[right] {$\omega$} coordinate (d);
\fill (d) circle (1pt);
\path (-1,7)  node[right] {$\beta$} coordinate (f);
\fill (f) circle (1pt);
\path (1,7)  node[right] {$\alpha$} coordinate (g);
\fill (g) circle (1pt);
\path (3,7)  node[right] {$a_{1}$} coordinate (h);
\fill (h) circle (1pt);
\path (4,7)  node[right] {$a_{2}$} coordinate (i);
\fill (i) circle (1pt);
\path (5,7)  node[right] {$a_{3}$} coordinate (j);
\fill (j) circle (1pt);
\path (4,1)  node[right] {$a$} coordinate (k);
\fill (k) circle (1pt);
\path (5.5,7)  node[right] {$$} coordinate (l);
\fill (l) circle (0.1pt);
\path (7,7)  node[right] {$$} coordinate (m);
\fill (m) circle (0.1pt);
\draw  (a) -- (b)  (b) -- (c) (d) -- (f) (d) -- (g) (a) -- (h) (b) -- (i) (c) -- (j) (k) -- (f) (k) -- (h) (k) -- (i) (k) -- (j) (k) -- (g);
\draw[densely dashed] (c)--(d) (l)--(m);
\end{tikzpicture}
\\Fig. 4. $M$.
\end{figure}

\begin{theorem}\label{strong d reflection}
 There exists a $T_{0}$ space $X$ admitting no strong $d$-reflection.
\begin{proof}
Let $Q$ be the space in Example \ref{b}. From Example \ref{b} and Example \ref{c}, we have $(Q,\tau_{Q})\notin \mathbf{S}$-$\mathbf{Top}_{d}$ and $(P,\tau_{P})\in \mathbf{S}$-$\mathbf{Top}_{d}$. We show that the strong $d$-reflection of $Q$ does not exist below.

Assume for the contradiction that the strong $d$-reflection of $Q$ exists. Let $\langle Q^{s-d}, \alpha_{Q}\rangle$ be the strong $d$-reflection. Define $f:(Q,\tau_{Q})\rightarrow (P,\tau_{P})$ by $f(x)=x$ for any $x\in Q$. It is obvious that $f$ is continuous. This implies that there exists a unique continuous mapping $f^{*}: Q^{s-d}\rightarrow (P,\tau_{P})$ such that $f=f^{*}\circ \alpha_{Q}$ by the definition of strong $d$-reflection. i.e., the following diagram commutes:

\centerline{\xymatrix{Q \ar[dr]_{f} \ar[r]^{\alpha_{Q}}
                & Q^{s-d} \ar[d]^{f^{*}}  \\
                &  P }}

We will prove that $(P,\tau_{P})$ and $Q^{s-d}$ are homeomorphic. Since $(Q,\tau_{Q})$ is a $d$-space and $\alpha_{Q}$ is continuous, we have $\{\alpha_{Q}(n)\}_{n\in \mathbb{N}}$ is a chain and $\alpha_{Q}(\omega)=\sup_{n\in \mathbb{N}}\alpha_{Q}(n)$.

    {\bf Claim 1}: $\ua \alpha_{Q}(\omega)\cap \ua \alpha_{Q}(a)\cap cl(\alpha_{Q}(Q))=\bigcap_{n\in \mathbb{N}}\ua \alpha_{Q}(n)\cap \ua \alpha_{Q}(a)\cap cl(\alpha_{Q}(Q))\neq \emptyset$.

    Suppose that $\ua \alpha_{Q}(\omega)\cap \ua \alpha_{Q}(a)\cap cl(\alpha_{Q}(Q))=\emptyset$. This means that $\bigcap_{n\in \mathbb{N}}\ua \alpha_{Q}(n)\cap \ua \alpha_{Q}(a)\subseteq Q^{s-d}\backslash cl(\alpha_{Q}(Q))$. It follows that there exists $n_{0}\in \mathbb{N}$ such that $\ua \alpha_{Q}(n_{0})\cap \ua \alpha_{Q}(a)\subseteq Q^{s-d}\backslash cl(\alpha_{Q}(Q))$ since $Q^{s-d}$ is a strong $d$-space, which contradicts $\alpha_{Q}(a_{n_{0}})\in \ua \alpha_{Q}(n_{0})\cap \ua \alpha_{Q}(a)\cap \alpha_{Q}(Q) $.

    Pick $b\in \ua \alpha_{Q}(\omega)\cap \ua \alpha_{Q}(a)\cap cl(\alpha_{Q}(Q))=\bigcap_{n\in \mathbb{N}}\ua \alpha_{Q}(n)\cap \ua \alpha_{Q}(a)\cap cl(\alpha_{Q}(Q)) $.

    {\bf Claim 2}: $f^{*}(b)=\beta$.

    Assume that there exists $x\in Q$ such that $f^{*}(b)=x$. By Claim 1, we have $\alpha_{Q}(\omega),\alpha_{Q}(a)\leq b$, this implies that $f^{*}(\alpha_{Q}(\omega))=f(\omega)=\omega\leq f^{*}(b)=x$ and $f^{*}(\alpha_{Q}(a))=f(a)=a\leq f^{*}(b)=x$ which contradicts that $\{\omega, a\}^{u}\cap Q=\emptyset$, where $\{\omega,a\}^{u}$ represents the set of all upper bound of $\{\omega,a\}$. So we have $f^{*}(b)=\beta$. Hence, $f^{*}$ is a surjection.

    Define $\alpha_{Q}^{*}:(P,\tau_{P})\rightarrow Q^{s-d}$ by
   \[ \alpha_{Q}^{*}(x)=
   \begin{cases}
\alpha_{Q}(x)& \mbox{if } x\in Q; \\
b & \mbox{if } x=\beta. \\
\end{cases}\]

  {\bf Claim 3}: $\alpha_{Q}^{*}$ is continuous.

  For any $U\in \mathcal{O}(Q^{s-d})$, there are two cases to consider:

  Case 1, $b\in U$. Then $(\alpha_{Q}^{*})^{-1}(U)=\{\beta\}\cup (\alpha_{Q})^{-1}(U)$. By the continuity of $\alpha_{Q}$, we know $(\alpha_{Q})^{-1}(U){\in} \tau_{Q}$. This implies that $\{a_{n}{\mid }n\in \mathbb{N}\}\backslash (\alpha_{Q})^{-1}(U) ~\mathrm{is}~ \mathrm{finite}$. So $\{a_{n}\mid n\in \mathbb{N}\}\backslash (\{\beta\}\cup (\alpha_{Q})^{-1}(U)) ~\mathrm{is}~ \mathrm{finite}$. To prove $(\alpha_{Q}^{*})^{-1}(U)\in \tau_{P}$,  it remains to prove that $\{\beta\}\cup (\alpha_{Q})^{-1}(U)$ is Scott open in $P$. Since $\ (\alpha_{Q})^{-1}(U)$ is Scott open in $Q$, we have $\{\beta\}\cup (\alpha_{Q})^{-1}(U)$ is Scott open in $P$.

  Case 2, $b\notin U$. Then $(\alpha_{Q}^{*})^{-1}(U)=(\alpha_{Q})^{-1}(U)$. To prove $(\alpha_{Q}^{*})^{-1}(U)\in \tau_{P}$, it remains to prove that $(\alpha_{Q})^{-1}(U)$ is an upper set. Let $x\leq y\in P$, $x\in (\alpha_{Q})^{-1}(U)$. If $y\in Q$, then $y\in (\alpha_{Q})^{-1}(U)$ since $(\alpha_{Q})^{-1}(U)$ is an upper set of $Q$. Otherwise, $y=\beta$. This implies that $x\in \da \beta\backslash \{\beta\}=\{a,\omega\}\cup \mathbb{N}$. It follows that $b\geq \alpha_{Q}(x)\in U$. Thus $b\in U$, which contradicts $b\notin U$.

  From the definition of $\alpha_{Q}^{*}$, we know $\alpha_{Q}=\alpha_{Q}^{*}\circ f$. Note that $f=f^{*}\circ \alpha_{Q}$. It follows that $\alpha_{Q}=\alpha_{Q}^{*}\circ f=\alpha_{Q}^{*}\circ f^{*}\circ \alpha_{Q}$ and $f=f^{*}\circ \alpha_{Q}=f^{*}\circ \alpha_{Q}^{*}\circ f$. Hence, $\alpha_{Q}^{*}\circ f^{*}=id_{Q^{s-d}}$ and $f^{*}\circ \alpha_{Q}^{*}=id_{P}$. By Claim 3, we have $(P,\tau_{P})$ and $Q^{s-d}$ are homeomorphic.

  {\bf Claim 4}: $Q$ admits no strong $d$-reflection.

  By Example \ref{d}, we have $(M,\tau_{M})\in\mathbf{S}$-$\mathbf{Top}_{d}$. Define $g:(Q,\tau_{Q})\rightarrow (M,\tau_{M})$ by $g(x)=x$ for any $x\in Q$. Obviously, $g$ is continuous. Let $i:(P,\tau_{P})\rightarrow (M,\tau_{M})$ be the inclusion map. Then $i$ is continuous. Define $h: (P,\tau_{P})\rightarrow (M,\tau_{M})$ by

    \begin{center}
  $$ h(x)=\left\{
\begin{aligned}
x &,&x\in Q; \\
\alpha &, & x=\beta. \\
\end{aligned}
\right.
$$
  \end{center}
  We need to prove that $h$ is continuous. For any $U\in \tau_{M}$, we need to distinguish among the four cases:

  Case 1, $\alpha,\beta\notin U$. Then $h^{-1}(U)=U\subseteq \{a_{n}\mid n\in \mathbb{N}\}$. This means that $U\in \tau_{P}$.

  Case 2, $\alpha\in U, \beta\notin U$. Then $h^{-1}(U)=(U\cup \{\beta\})\cap P$ and $U\subseteq \{\alpha\}\cup\{a_{n}\mid n\in \mathbb{N}\}$. Hence, $h^{-1}(U)\in \tau_{P}$.

  Case 3, $\alpha\notin U,\beta\in U$. Then $h^{-1}(U)=U\backslash\{\beta\}$ and $U\subseteq \{\beta\}\cup\{a_{n}\mid n\in \mathbb{N}\}$. It follows that $h^{-1}(U)\in \tau_{P}$.

  Case 4, $\alpha,\beta\subseteq U$. Then $h^{-1}(U)=U\backslash \{\alpha\}=U\cap P$. It is worth nothing that $P$ is a Scott closed subset of $M$, then we have $h^{-1}(U)\in \sigma(P)$. Therefore, $h^{-1}(U)\in \tau_{P}$.

  Since $g=i\circ f=h\circ f$, we have $i=h$, which contradicts $i\neq h$.
 \end{proof}
\end{theorem}

By Theorem \ref{strong d reflection}, we obtain the following theorem immediately.

\begin{theorem}
The category $\mathbf{S}$-$\mathbf{Top}_{d}$ is not a reflective subcategory of $\mathbf{Top}_{0}$.
\end{theorem}

\section{On some questions about super $\mathrm{H}$-sober spaces}
In this section, we will prove that for each irreducible subset system ($R$-subset
system for short) $\mathrm{H}$, property $M$ mentioned in \cite{Xiaoquan21} naturally holds. Based on this result, we get that Question 1.1$\sim$Question 1.7 hold. Furthermore, we generalize some results in \cite{Xiaoquan21}.

The category of all sets with mappings is denoted by $\mathbf{Set}$.
\begin{definition}\cite{Xiaoquan21}
A covariant functor $\mathrm{H}:\mathbf{Top}_{0}\longrightarrow \mathbf{Set}$ is called a \emph{subset system} on $\mathbf{Top}_{0}$ provided that the following two conditions are satisfied:
\begin{enumerate}[(1)]
\item $S(X)\subseteq \mathrm{H}(X)\subseteq 2^{X}$ (the set of all subsets of X) for each $X\in ob(\mathbf{Top}_{0}$).
\item For any continuous mapping $f:X\rightarrow Y$ in $\mathbf{Top}_{0}$, $\mathrm{H}(f)(A)=f(A)\in \mathrm{H}(Y)$ for all $A\in \mathrm{H}(X)$.
\end{enumerate}

For a subset system $\mathrm{H}:\mathbf{Top}_{0}\longrightarrow \mathbf{Set}$ and a $T_{0}$ space $X$, let $\mathrm{H}_{c}(X)=\{\overline{A}\mid A\in \mathrm{H}(X)\}$. We call $A\subseteq X$ an \emph{$\mathrm{H}$-set} if $A\in \mathrm{H}(X)$. The sets in $\mathrm{H}_{c}(X)$ are called \emph{closed $\mathrm{H}$-sets}.
\end{definition}

\begin{definition}\cite{Xiaoquan21}
A subset system $\mathrm{H}:\mathbf{Top}_{0}\longrightarrow \mathbf{Set}$ is called an \emph{irreducible subset system}, or an \emph{R-subset
system} for short, if $\mathrm{H}(X)\subseteq \mathrm{Irr}(X)$ for all $X\in ob(\mathbf{Top}_{0})$. The family of all $R$-subset systems is denoted by $\mathcal{H}$. Define a partial order $\leq$ on $\mathcal{H}$ by $\mathrm{H}_{1}\leq \mathrm{H}_{2}$ iff $\mathrm{H}_{1}(X)\subseteq \mathrm{H}_{2}(X)$ for all $X\in ob(\mathbf{Top}_{0})$.
\end{definition}

\begin{definition}\cite{Xiaoquan21}
Let $\mathrm{H}:\mathbf{Top}_{0}\longrightarrow \mathbf{Set}$ be an R-subset
system. A $T_{0}$ space $X$ is called \emph{$\mathrm{H}$-sober} if for any $A\in \mathrm{H}(X)$, there is a (unique) point $x\in X$ such that $\overline{A} = \overline{\{x\}}$ or, equivalently, if $\mathrm{H}_{c}(X)=S_{c}(X)$. The category of all $\mathrm{H}$-sober spaces with continuous mappings is denoted by $\mathbf{H}$-$\mathbf{Sob}$.
\end{definition}

It is not difficult to verify that $\mathbf{Sob}\subseteq \mathbf{H}$-$\mathbf{Sob}$. By Definition \ref{k-reflection}, for $\mathbf{K}=\mathbf{H}$-$\mathbf{Sob}$, the $\mathbf{K}$-reflection of $X$ is called the \emph{$\mathrm{H}$-sober reflection} of $X$, or the \emph{$\mathrm{H}$-sobrification} of $X$ and $X^{h}$ to denote the space of $\mathrm{H}$-sobrification of $X$ if it exists. Moreover, a $\mathbf{K}$-determined set of $X$ in Definition \ref{k-determined} is called a \emph{$\mathrm{H}$-sober determined set} of $X$, Denote by $\mathrm{H}^{d}(X)$ the set of all $\mathrm{H}$-sober determined subsets of $X$. The set of all closed $\mathrm{H}$-sober determined subsets of $X$ is denoted by $\mathrm{H}^{d}_{c}(X)$.

\begin{theorem}\cite{Xiaoquan21}\label{$H$-sobrification}
Let $\mathrm{H}:\mathbf{Top}_{0}\longrightarrow \mathbf{Set}$ be an R-subset
system. Then $\mathbf{H}$-$\mathbf{Sob}$ is adequate. Therefore, for any $T_{0}$ space $X$, $X^{h}=P_{H}(\mathrm{H}^{d}_{c}(X))$ with the canonical topological embedding $\eta^{h}_{X}$: $X\rightarrow X^{h}$ is the $\mathrm{H}$-sobrification of $X$, where $\eta^{h}_{X}(x)=\overline{\{x\}}$ for all $x\in X$.
\end{theorem}

Let $\mathrm{H}:\mathbf{Top}_{0}\longrightarrow \mathbf{Set}$ be an R-subset
system. In \cite{Xiaoquan21}, it has been shown that $\mathbf{H}$-$\mathbf{Sob}$ is a full subcategory  of $\mathbf{Top}_{0}$ containing $\mathbf{Sob}$ and is closed with respect to homeomorphisms. Moreover, $\mathbf{H}$-$\mathbf{Sob}$ is adequate by Theorem \ref{$H$-sobrification}. Conversely, suppose that $\mathbf{K}$ is adequate and closed with respect to homeomorphisms. Is there an R-subset
system $\mathrm{H}$ such that $\mathbf{K}=\mathbf{H}$-$\mathbf{Sob}$? The answer is positive.

\begin{lemma}\label{2}
For a $T_{0}$ space, $S(X)\subseteq \mathrm{K}^{d}(X)\subseteq \mathrm{Irr}(X)$.
\end{lemma}

\begin{proof}
Clearly, $S(X)\subseteq \mathrm{K}^{d}(X)$. Suppose that $A\in \mathrm{K}^{d}(X)$. Consider the sobrification $X^{s}$ ($=P_{H}(\mathrm{Irr}_{c}(X))$) of $X$ and the canonical topological embedding $\eta_{X}:X\rightarrow X^{s}$ defined by $\eta_{X}(x)=\overline{\{x\}}$. Then there is a $B\in \mathrm{Irr}_{c}(X)$ such that $\overline{\eta_{X}(A)}=\overline{\{B\}}$. It is easy to check that $\overline{A}=B$. Hence, $A\in \mathrm{Irr}(X)$.
\end{proof}

\begin{theorem}
Suppose that $\mathbf{K}$ is adequate and closed with respect to homeomorphisms. Then for a covariant functor $\mathrm{K}:\mathbf{Top}_{0}\longrightarrow \mathbf{Set}$ defined by $\mathrm{K}^{d}(X)$, $\forall$ $X\in ob(\mathbf{Top}_{0})$, the following statements hold:
\begin{enumerate}[(1)]
\item $\mathrm{K}:\mathbf{Top}_{0}\longrightarrow \mathbf{Set}$ is an $R$-subset system.
\item $\mathbf{K}$=$\mathbf{K}$-$\mathbf{sober}$, that is for each $\mathbf{K}$-space $X$ and any $A\in \mathrm{K}^{d}(X)$, there exists a (unique) element $x\in X$ such that $\overline{A}=\overline{\{x\}}$.
\end{enumerate}
\end{theorem}

\begin{proof}
These  statements directly follow from Lemma \ref{1}, Lemma \ref{2} and Corollary \ref{3}.
\end{proof}

For a $T_{0}$ space $X$ and $\mathcal{K}\subseteq K(X)$, let $M(\mathcal{K})=\{A\in \Gamma(X)\mid A\cap K\neq \emptyset \mbox{ for all } K\in \mathcal{K}\}$ (that is, $\mathcal{K}\subseteq \diamondsuit A$) and $m(\mathcal{K})=\{A\in \Gamma(X)\mid A\mbox{ is a minimal member of } M(\mathcal{K})\}$ (\cite{Xiaoquan20}).

In \cite{Xiaoquan21}, Xu proposed that an $R$-subset system $\mathrm{H}:\mathbf{Top}_{0}\longrightarrow \mathbf{Set}$ is said to satisfy property $M$ if for any $T_{0}$ space $X$, $\mathcal{K}\in \mathrm{H}(P_{S}(X))$ and $A\in M(\mathcal{K})$, then $\{\ua (K\cap A)\mid K\in \mathcal{K}\}\in \mathrm{H}(P_{S}(X))$. Furthermore, he proved some conclusions under the assumption that $\mathrm{H}$ has property $M$. In the following, we find that property $M$ naturally holds for each $R$-subset system $\mathrm{H}$.

\begin{proposition}\label{property $M$}
Let $\mathrm{H}:\mathbf{Top}_{0}\longrightarrow \mathbf{Set}$ be an $R$-subset system. Then property $M$ holds.
\end{proposition}

\begin{proof}
Suppose that $X$ is a $T_{0}$ space, $\mathcal{K}\in \mathrm{H}(P_{S}(X))$ and $A\in M(\mathcal{K})$. Take a map $f:P_{S}(X)\rightarrow P_{S}(X)$ defined by
$$f(K)=\ua(K\cap A)$$
for any $K\in K(X)$. It is straightforward to check that $f$ is well-defined.

{\bf Claim}: $f$ is continuous.

For any $U\in \mathcal{O}(X)$, $f^{-1}(\Box U)=\{K\in K(X)\mid \ua(K\cap A)\in \Box U\}=\{K\in K(X)\mid \ua(K\cap A)\subseteq U\}=\Box((X\backslash A)\cup U)$. Hence, $f^{-1}(\Box U)$ is open in $K(X)$. So $f$ is continuous.

Since $\mathrm{H}$ is an $R$-subset system and $\mathcal{K}\in \mathrm{H}(P_{S}(X))$, we have that $f(\mathcal{K})=\{\ua (K\cap A)\mid K\in \mathcal{K}\}\in \mathrm{H}(P_{S}(X))$, that is, property $M$ holds.
\end{proof}

By Theorem 6.19, Theorem 6.20, Proposition 6.21, Corollary 6.22, Theorem 7.16 in \cite{Xiaoquan21} and Proposition \ref{property $M$}, we get that Question 1.1$\sim$Question 1.6 hold.

\begin{definition}\cite{Xiaoquan21}\label{super H-sober}
Let $\mathrm{H}:\mathbf{Top}_{0}\longrightarrow \mathbf{Set}$ be an $R$-subset system. A $T_{0}$ space $X$ is called \emph{super $\mathrm{H}$-sober} provided
its Smyth power space $P_{S}(X)$ is $\mathrm{H}$-sober. The category of all super $\mathrm{H}$-sober spaces with continuous mappings is denoted by $\mathbf{SH}$-$\mathbf{Sob}$.
\end{definition}

\begin{definition}\cite{Xiaoquan21}\label{5}
Let $\mathrm{H}:\mathbf{Top}_{0}\longrightarrow \mathbf{Set}$ be an $R$-subset system and $X$ a $T_{0}$ space. A nonempty subset $A$ of $X$ is said to have \emph{$\mathrm{H}$-Rudin property}, if there exists $\mathcal{K}\in \mathrm{H}(P_{S}(X))$ such that $\overline{A}\in m(\mathcal{K})$, that is, $\overline{A}$ is a minimal closed set that intersects all members of $\mathcal{K}$. Let $\mathrm{H}^{R}(X)=\{A\subseteq X \mid A \mbox{ has $\mathrm{H}$-Rudin property}\}$. The sets in $\mathrm{H}^{R}(X)$ will also be called $\mathrm{H}$-Rudin sets.
\end{definition}

\begin{lemma}\cite{Xiaoquan21}\label{7}
Let $\mathrm{H}:\mathbf{Top}_{0}\longrightarrow \mathbf{Set}$ be an $R$-subset system and $X$ a $T_{0}$ space. Then $\mathrm{H}(X)\subseteq \mathrm{H}^{R}(X)\subseteq \mathrm{Irr}(X)$.
\end{lemma}

\begin{lemma}\cite{Xiaoquan21}\label{4}
Let $\mathrm{H}:\mathbf{Top}_{0}\longrightarrow \mathbf{Set}$ be an $R$-subset system. Then $\mathrm{H^{R}}:\mathbf{Top}_{0}\longrightarrow \mathbf{Set}$ is an $R$-subset system, where for any continuous mapping $f:X\rightarrow Y$ in $\mathbf{Top}_{0}$, $\mathrm{H}^{R}(f):\mathrm{H}^{R}(X)\rightarrow \mathrm{H}^{R}(Y)$ is defined by $\mathrm{H}^{R}(f)(A)=f(A)$ for each $A\in \mathrm{H}^{R}(X)$.
\end{lemma}

\begin{theorem}\cite{Xiaoquan21}\label{11}
Let $\mathrm{H}:\mathbf{Top}_{0}\longrightarrow \mathbf{Set}$ be an $R$-subset system and $X$ a $T_{0}$ space. Then the following conditions are equivalent:
\begin{enumerate}[(1)]
\item $X$ is super $\mathrm{H}$-sober.
\item $X$ is $\mathrm{H}^{R}$-sober.
\end{enumerate}
\end{theorem}

\begin{definition}\cite{Xiaoquan21}\label{property Q}
An $R$-subset system $\mathrm{H}:\mathbf{Top}_{0}\longrightarrow \mathbf{Set}$ is said to satisfy \emph{property $Q$} if for any $\mathcal{K}\in \mathrm{H}(P_{S}(X))$ and any $A\in M(\mathcal{K})$, $A$ contains a closed $\mathrm{H}$-set $C$ such that $C\in M(\mathcal{K})$.
\end{definition}

In \cite{Xiaoquan21}, Theorem 5.12 pointed out that if $X$ is $\mathrm{H}$-sober and $\mathrm{H}$ has property $Q$, then $X$ is super $\mathrm{H}$-sober. The following Theorem will show that the converse also holds.

\begin{theorem}
Let $\mathrm{H}:\mathbf{Top}_{0}\longrightarrow \mathbf{Set}$ be an $R$-subset system. For a $T_{0}$ space $X$, the following two conditions are equivalent:
\begin{enumerate}[(1)]
\item $X$ is $\mathrm{H}$-sober and $\mathrm{H}$ has property $Q$.
\item $X$ is super $\mathrm{H}$-sober.
\end{enumerate}
\end{theorem}

\begin{proof}
$(1)\Rightarrow(2)$: Follows directly from Theorem 5.12 in \cite{Xiaoquan21}.

$(2)\Rightarrow(1)$: We only need to show that $\mathrm{H}$ has property $Q$. For any $\mathcal{K}\in \mathrm{H}(P_{S}(X))$ and any $A\in M(\mathcal{K})$, by Rudin Lemma, we have that there exists a minimal closed subset $B\subseteq A$ such that $B\in M(\mathcal{K})$. Thus $B\in \mathrm{H}^{R}(X)$. Since $X$ is super $\mathrm{H}$-sober, by Theorem \ref{11}, we have that $X$ is $\mathrm{H}^{R}$-sober. Therefore, there exists $x\in X$ such that $B=\overline{\{x\}}$. This implies that $B$ is a closed $\mathrm{H}$-set. Thus $\mathrm{H}$ has property $Q$.
\end{proof}

In the following, we will prove that for an $R$-subset system $\mathrm{H}:\mathbf{Top}_{0}\longrightarrow \mathbf{Set}$, $\mathrm{R}:\mathcal{H}\rightarrow \mathcal{H}$, defined by $\mathrm{H}\mapsto \mathrm{H}^{R}$ is a closure operator.

\begin{theorem}
Let $\mathrm{H}:\mathbf{Top}_{0}\longrightarrow \mathbf{Set}$ be an $R$-subset system. Then $\mathrm{R}:\mathcal{H}\rightarrow \mathcal{H}$, defined by $\mathrm{H}\mapsto \mathrm{H}^{R}$ is a closure operator.
\end{theorem}

\begin{proof}
For any $R$-subset system $\mathrm{H}$, by Lemma \ref{4}, we have that $\mathrm{H}^{R}$ is also an $R$-subset system. So $\mathrm{R}$ is well-defined.

{\bf Claim 1}: $\mathrm{R}$ is monotone, that is, for $\mathrm{H_{1}},\mathrm{H_{2}}\in \mathcal{H}$ and $\mathrm{H_{1}}\leq\mathrm{H_{2}}$, $\mathrm{H}^{R}_{1}\leq \mathrm{H}^{R}_{2}$.

Assume that $X$ is a $T_{0}$ space. For any $A\in \mathrm{H}^{R}_{1}(X)$, by Definition \ref{5}, there exists $\mathcal{K}\in \mathrm{H_{1}}(P_{S}(X))$ such that $\overline{A}\in m(\mathcal{K})$. Since $\mathrm{H_{1}}\leq\mathrm{H_{2}}$, that is, $\mathrm{H_{1}}(P_{S}(X))\subseteq\mathrm{H_{2}}(P_{S}(X))$, we have that $\mathcal{K}\in \mathrm{H_{2}}(P_{S}(X))$ and hence, $A\in \mathrm{H}^{R}_{2}(X)$. So $\mathrm{H}^{R}_{1}(X)\subseteq \mathrm{H}^{R}_{2}(X)$ for any $T_{0}$ space $X$, which implies that $\mathrm{H}^{R}_{1}\leq \mathrm{H}^{R}_{2}$.

{\bf Claim 2}: $R$ is expansive, that is, for any $\mathrm{H}\in \mathcal{H}$, $\mathrm{H}\leq \mathrm{H}^{R}$.

Follows directly from Lemma \ref{7}.

{\bf Claim 3}: $R$ is idempotent, that is, for any $\mathrm{H}\in \mathcal{H}$, $(\mathrm{H}^{R})^{R}=\mathrm{H}^{R}$.

Since $R$ is expansive, we have that $\mathrm{H}^{R}\leq (\mathrm{H}^{R})^{R}$. Conversely, we only need to show that $(\mathrm{H}^{R})^{R}(X)\subseteq \mathrm{H}^{R}(X)$ for any $T_{0}$ space $X$. Suppose that $A\in (\mathrm{H}^{R})^{R}(X)$. By Definition \ref{5}, there exists $\mathcal{A}\in \mathrm{H}^{R}(P_{S}(X))$ such that $\overline{A}\in m(\mathcal{A})$. For $\mathcal{A}$, again by Definition \ref{5}, there exists $\mathcal{K}=\{\mathcal{K}_{i}\}_{i\in I}\in \mathrm{H}(P_{S}(P_{S}(X)))$ such that $\overline{\mathcal{A}}\in m(\mathcal{K})$. For any $i\in I$, let $K_{i}=\bigcup \ua_{K(X)}(\mathcal{K}_{i}\bigcap \overline{\mathcal{A}})=
\bigcup(\mathcal{K}_{i}\bigcap \overline{\mathcal{A}})$. Then by Lemma \ref{9}, Corollary \ref{8} and property $M$ of $\mathrm{H}$, we have $\{K_{i}\}_{i\in I}\in \mathrm{H}(P_{S}(X))$ and $K_{i}\in \overline{\mathcal{A}}$ for each $i\in I$ since $\overline{\mathcal{A}}$ is a lower set.

{\bf Claim 3.1}: $\overline{A}\in M(\{K_{i}\}_{i\in I})$.

Suppose not, there exists $i\in I$ such that $\overline{A}\cap K_{i}=\emptyset$. That is $K_{i}\subseteq X\backslash \overline{A}$. So $K_{i}\in \Box(X\backslash \overline{A})$. Since $K_{i}\in \overline{\mathcal{A}}$, there exists $K\in \Box(X\backslash \overline{A})\bigcap \mathcal{A}$, which contradicts the fact that $\overline{A}\in m(\mathcal{A})$.

{\bf Claim 3.2}: $\overline{A}\in m(\{K_{i}\}_{i\in I})$.

Let $B\subseteq \overline{A}$ be a closed subset and $B\in M(\{K_{i}\}_{i\in I})$. Then for any $i\in I$, there exists an element $N\in \mathcal{K}_{i}\bigcap \overline{\mathcal{A}}$ such that $B\cap N\neq \emptyset$, it is $\diamondsuit B\bigcap \mathcal{K}_{i}\bigcap \overline{\mathcal{A}}\neq \emptyset$. By the minimality of $\overline{\mathcal{A}}$, we have $\overline{\mathcal{A}}=\diamondsuit B\bigcap \overline{\mathcal{A}}$, and consequently, $\mathcal{A}\subseteq\overline{\mathcal{A}}\subseteq\diamondsuit B$. So $B\in M(\mathcal{A})$. By the minimality of $\overline{A}$, we have $\overline{A}=B$. Thus $\overline{A}\in m(\{K_{i}\}_{i\in I})$.

Therefore, for $A$, there exists $\{K_{i}\}_{i\in I}\in \mathrm{H}(P_{S}(X))$ such that $A\in m(\{K_{i}\}_{i\in I})$, which means $A\in \mathrm{H}^{R}(X)$.
\end{proof}

\section{The finite product of hyper-sober spaces}
The concept of hyper-sober spaces was introduced in \cite{Zhao15}. And in \cite{Wennana}, Wen and Xu gave a counterexample to show that the product of a countable infinite family of hyper-sober spaces is not hyper-sober in general. Meantime, they posed the following question:

Is the product space of two hyper-sober spaces again a hyper-sober space?

In this Section, we will give a positive answer to the above question.

\begin{definition}(\cite{Zhao15})
A topological space $X$ is called \emph{hyper-sober} if for any irreducible set $F$, there is a unique $x\in F$ such that $F\subseteq \mathrm{cl}(\{x\})$.
\end{definition}

\begin{lemma}(\cite{Xiaoquan21})\label{13}
Let $X$ be a space. Then the following conditions are equivalent for a subset $A\subseteq X$:
\begin{enumerate}[(1)]
\item $A$ is an irreducible subset of $X$.
\item $\mathrm{cl}_{X}(A)$ is an irreducible subset of $X$.
\end{enumerate}
\end{lemma}

\begin{lemma}(\cite{Xiaoquan21})\label{12}
If $f:X\rightarrow Y$ is continuous and $A\in \mathrm{Irr}(X)$, then $f(A)\in \mathrm{Irr}(Y)$.
\end{lemma}

\begin{theorem}
Let $X$ and $Y$ be two hyper-sober spaces. Then the product space $X\times Y$ is also hyper-sober.
\end{theorem}

\begin{proof}
Let $A$ be an irreducible subset in $X\times Y$. Suppose $P_{X}:X\times Y\rightarrow X$ and $P_{Y}:X\times Y\rightarrow Y$ are projections, respectively. Note that $P_{X}$ and $P_{Y}$ are continuous. By Lemma \ref{12}, $P_{X}(A)\in \mathrm{Irr}(X)$ and $P_{Y}(A)\in \mathrm{Irr}(Y)$. Since $X$ and $Y$ are hyper-sober, there exist $x\in P_{X}(A)$ and $y\in P_{Y}(A)$ such that $P_{X}(A)\subseteq \mathrm{cl}(\{x\})$ and $P_{Y}(A)\subseteq \mathrm{cl}(\{y\})$. This implies that $A\subseteq P_{X}(A)\times P_{Y}(A)\subseteq \mathrm{cl}(\{x\})\times \mathrm{cl}(\{y\})=\da(x,y)$. It is sufficient to prove that  $(x,y)\in A$.

{\bf Claim}: $x\not\in \overline{P_{X}(A)\backslash\{x\}}$ and $y\not\in \overline{P_{Y}(A)\backslash\{y\}}$.

Suppose not, $x\in \overline{P_{X}(A)\backslash\{x\}}$. One can directly get $\overline{P_{X}(A)\backslash\{x\}}=\da x$. Then $P_{X}(A)\backslash\{x\}\in \mathrm{Irr}(X)$ by Lemma \ref{13}. Again since $X$ is hyper-sober, there is an element $a\in P_{X}(A)\backslash\{x\}$ such that $P_{X}(A)\backslash\{x\}\subseteq \da a$. This implies that $x\in X\setminus \da a$. Thus $(P_{X}(A)\backslash\{x\})\cap(X\setminus \da a)\neq \emptyset$, which contradicts $P_{X}(A)\backslash\{x\}\subseteq \da a$. So $x\not\in \overline{P_{X}(A)\backslash\{x\}}$. For $y\not\in \overline{P_{Y}(A)\backslash\{y\}}$, the proof is similar to that the case $x\not\in \overline{P_{X}(A)\backslash\{x\}}$.

Therefore, there exist open neighborhoods $U$ of $x$ and $V$ of $y$ such that $U\cap (P_{X}(A)\backslash\{x\})=\emptyset$ and $V\cap (P_{Y}(A)\backslash\{y\})=\emptyset$, respectively. Since $(x,y)\in U\times V$ and $(x,y)\in\mathrm{cl}(A)$, there exists $(b,c)\in (U\times V)\cap A$. This implies that $b\in U\cap P_{X}(A)$ and $c\in V\cap P_{Y}(A)$. So $b=x$ and $c=y$, and hence, $(x,y)\in A$.
\end{proof}

\section{Reference}
\bibliographystyle{plain}

\begin{thebibliography}{10}
\bibitem{Barr}
M. Barr, C. Wells,
\newblock {Category Theory}:
\newblock {Lecture Notes for ESSLLI}, 1999.

\bibitem{Davey02}
B.A. Davey,
\newblock {Introduction to Lattices and Order, 2nd edition},
\newblock {\em Cambridge University Press}, 2003.

\bibitem{Ershov 17}
Y. Ershov,
\newblock {The $d$-rank of a topological space},
\newblock {\em Algebra and Logic}, 56.2 (2017): 98-107.

\bibitem{clad3}
G.~Gierz, K.~H. Hofmann, K.~Keimel, J.~D. Lawson, M.~Mislove, and D.~S. Scott,
\newblock { Continuous Lattices and Domains}, Volume~93 of {\em Encyclopedia
  of Mathematics and its Applications}.
\newblock Cambridge University Press, 2003.

\bibitem{nht2}
J.~Goubault-Larrecq,
\newblock Non-Hausdorff Topology and Domain Theory, Volume~22 of {\em New
  Mathematical Monographs},
\newblock {\em Cambridge University Press}, 2013.

\bibitem{X Jia}
X. Jia, A. Jung,
\newblock {A note on coherence of dcpos},
\newblock {\em Topol. Appl.}, 209 (2016): 235-238.

\bibitem{R. Heckmann}
R. Heckmann,
\newblock {An upper power domain construction in terms of strongly compact sets},
\newblock {\em Springer-Verlag Lecture Notes in Computer Science.}, 598 (1992): 272-293.

\bibitem{Heckmann}
R. Heckmann, K. Keimel,
\newblock {Quasicontinuous domains and the Smyth powerdomain},
\newblock {\em Electron. Notes Theor. Comput. Sci.}, 298 (2013): 215-232.

\bibitem{Xiaoquan20}
X. Xu,
\newblock {A direct approach to $\mathbf{K}$-reflections of $T_{0}$ spaces},
\newblock {\em Topol. Appl.}, 272 (2020) 107076.

\bibitem{Zhao Xu}
X. Xu, D. Zhao,
\newblock {On topological Rudin's Lemma, well-filtered spaces and sober spaces},
\newblock {\em Topol. Appl.}, 272 (2020) 107080.

\bibitem{Xiaoquan}
X. Xu, D. Zhao,
\newblock {Some open problems on well-filtered spaces and sober spaces},
\newblock {\em Topol. Appl.}, (2020) 107540.

\bibitem{Xiaoquan21}
X. Xu,
\newblock {On $\mathrm{H}$-sober spaces and $\mathrm{H}$-sobrifications of $T_{0}$ spaces},
\newblock {\em Topol. Appl.}, 289 (2021) 107548.

\bibitem{Zhao15}
D. Zhao, W. Ho,
\newblock {On topologies defined by irreducible sets},
\newblock {\em J. Log. Algebr. Methods Program.}, 84.1 (2015): 185-195.

\bibitem{Wennana}
N. Wen, X. Xu,
\newblock {Some basic properties of hyper-sober spaces} (in Chinese),
\newblock {\em Fuzzy Systems and Mathematics}, accepted.

\end{thebibliography}

\end{document}